\documentclass{amsart}
\usepackage{amssymb,amsmath,mathrsfs,amsfonts}
\usepackage{amscd, amssymb, amsmath, amsthm, graphics}
\usepackage{amsmath,amsfonts,amsthm,amssymb}
\usepackage{fancyhdr}
\usepackage[all]{xy}
\usepackage[colorlinks=true]{hyperref}

\newtheorem{theorem}{Theorem}[section]
\newtheorem{lemma}[theorem]{Lemma}
\newtheorem{proposition}[theorem]{Proposition}
\newtheorem{corollary}[theorem]{Corollary}

\theoremstyle{definition}

\newcommand{\SO}{{\mathrm{SO}}}

\newcommand{\Integral}{{\mathbb{Z}}}

\begin{document}
\sloppy

\title{Virtual 1-domination of 3-manifolds}

\author{Yi Liu}
\address{Beijing International Center of Mathematical Research, Peking University, Beijing 100871 China}
\email{liuyi@math.pku.edu.cn}
\author{Hongbin Sun}
\address{Department of Mathematics, UC Berkeley, CA 94720, USA}
\email{hongbins@math.berkeley.edu}


\thanks{The first author is supported by the Recruitment Program of Global Youth Experts of China.
The second author is partially supported by Grant No.~DMS-1510383 of the National Science Foundation of the United States.}

\date{\today}
\begin{abstract}
	It is shown in this paper that given any closed oriented hyperbolic 3-manifold, every closed oriented 3-manifold
	is mapped onto by a finite cover of that manifold via a map of degree 1,
	or in other words, virtually 1-dominated by that manifold.
	This improves a known result of virtual 2-domination.
	The proof invokes a recently developed enhanced version of the connection principle
	in good pants constructions.
\end{abstract}

\maketitle

\section{Introduction}
	Over every $3$-manifold, there is a fantastic world of finite covers.
	Virtual objects can be observed in it and hidden properties can be revealed.
	Speaking of closed hyperbolic $3$-manifolds in general,	it is only for a few years
	that people have been able to
	explore that world with powerful devices.
	
	One of the devices is virtual specialization of finite-volume hyperbolic $3$-manifold groups.
	It has been completed by the celebrated Virtual Haken Theorem and Virtual Fibering Theorem due to I.~Agol \cite{Ag},
	based on the work of D.~Wise \cite{Wi} on quasiconvex hierarchies of word-hyperbolic special cube complexes.
	As a very strong consequence, the fundamental group of any finite-volume hyperbolic $3$-manifold is LERF.
	
	The package of good pants constructions is
	our new ruler and compass.
	As the LERF property enables us to separate subgroups,
	good pants constructions provide us a robust method
	to produce interesting ones,
	in closed hyperbolic $3$-manifold groups.
	In \cite{KM}, J.~Kahn and V.~Markovic has invented a way to
	construct a $\pi_1$-injectively immersed closed subsurface in a closed hyperbolic $3$-manifold,
	by pasting nearly totally geodesic regular pants in a nearly totally geodesic manner.
	In subsequent works of various authors,
	the construction is extended to build $\pi_1$-injectively immersed subsurfaces
	with prescribed boundary and relative homology class \cite{LM}, to build $\pi_1$-injectively 
    immersed subsurfaces in rank one locally symmetric spaces with $R$-varied good pants \cite{Ha},
	and to build certain $\pi_1$-injectively immersed $2$-complexes in a closed hyperbolic $3$-manifold
	for interesting applications \cite{Sun1,Sun2}.

	In this paper, we employ the developed techniques to prove the following theorem:
	
	\begin{theorem}\label{main}
		For any closed oriented hyperbolic $3$-manifold $M$, and any closed oriented $3$-manifold $N$,
		there exist a finite cover $M'$ of $M$ and a degree-$1$ map
		$f\colon M'\rightarrow N$.
		In other words, every closed oriented hyperbolic $3$-manifold virtually $1$-dominates
		any other closed oriented $3$-manifold.
	\end{theorem}
	
	By an earlier work of the second author \cite{Sun2},
	it has been proved that every closed oriented hyperbolic $3$-manifold
	virtually $2$-dominates any other closed oriented $3$-manifold.
	The degree $2$ there is taken so as to overcome
	a known $\Integral_2$--valued obstruction, which arises naturally in good pants constructions \cite{LM}.
	By invoking an enhanced connection principle, introduced lately by the first author \cite{Liu},
	we are able to directly maintain the vanishing of the obstructions
	in the course of our construction.
	The new construction leads to
	the promotion of the earlier result to Theorem \ref{main}.
	
	It is also known that every closed hyperbolic $3$-manifold $M$ has many
	virtual homological torsion	elements.
	In fact, every finite abelian group $A$ can be embedded into $H_1(M';\Integral)$ as a direct summand
	for	some finite cover $M'$ of $M$ \cite{Sun1}.
	Since every degree-$1$ map induces an embedding as a direct summand on cohomology,
	the result of \cite{Sun1} is also implied by Theorem \ref{main},
	taking $N$ to be a direct sum of lens spaces.
	Note that virtual $2$-domination only implies an embedding of $A$,
	but not necessarily as a direct summand.

	\begin{corollary}\label{kdomination}
		For any closed oriented hyperbolic $3$-manifold $M$, and any closed oriented $3$-manifold $N$, and any integer $k$,
		there exist a finite cover $M'$ of $M$ and a $\pi_1$--surjective degree-$k$ map $f\colon M'\to N$.
		In other words, $M$ virtually $\pi_1$-surjectively $k$-dominates $N$.
	\end{corollary}

	\begin{proof}
		Take two positive integers $p$ and $q$ such that $k$ equals $p-q$. Denote by
			$$N'=N^{\#p} \# (-N)^{\#q}$$
		the closed oriented $3$-manifold which is the connected sum of
		$p$ copies of $N$ and $q$ copies of the orientation-reversal of $N$.
		There is a map $N'\to N$ induced by identifying the factors with $N$ in an obvious fashion,
		which is $\pi_1$--surjective and degree-$k$.
		The virtual $1$-domination $M'\to N'$ by Theorem \ref{main}
		composed with the map $N'\to N$ yields a virtual $\pi_1$--surjective $k$-domination
		as claimed.
	\end{proof}

	Recently G.~Lakeland and C.~Leininger have constructed strict contractions
	between hyperbolic $3$-orbifolds and hyperbolic $4$-orbifolds \cite{LL}.
	Here a strict contraction refers to a degree-$1$ $\kappa$-Lipschitz map where $\kappa<1$.
	A closer examination of our construction yields the following metric refinement of Theorem \ref{main},
	which implies the existence of virtual strict contractions onto any Riemannian $N$,
	with respect to the hyperbolic metric of $M$.
	
	\begin{theorem}\label{contraction}
		For any closed oriented hyperbolic $3$-manifold $M$,
		and any closed oriented Riemannian $3$-manifold $N$,
		and any constant $\kappa>0$,
		there exist a finite cover $M'$ of $M$
		and a degree-$1$ $\kappa$-Lipschitz map $f\colon M'\rightarrow N$.
	\end{theorem}
	
	In the rest of the introduction, we point out a few major differences
	between our construction of virtual $1$-dominations (Theorem \ref{main})
	and the earlier construction of virtual $2$-dominations by \cite{Sun2}.
	The reader is referred to the introduction of Section \ref{Sec-hypCase}
	for a more detailed outline of the proof.
	
	First, we abandon use of a 3-dimensional
	pseudo-manifold or orbifold as an intermediate stage of the construction.
	Instead, the bulk of our construction produces a virtual $1$-domination
	onto a closed oriented hyperbolic $3$-manifold
	that $1$-dominates the considered target.
	The intermediate hyperbolic $3$-manifold can be fairly general,
	as we only need to extract information from an arbitrary geometric triangulation
	of it to instruct our construction.
	Not only is the new intermediate object more familiar in appearance,
	but it allows us to properly speak of certain associated structures,
	such as the frame bundle and the homology.
	
	Secondly, as mentioned,
	our construction requires careful control of the $\Integral_2$--valued obstructions.
	To be more precise, we follow a similar strategy as \cite{Sun2}	to build
	a $\pi_1$-injectively immersed quasiconvex $2$-subcomplex in the domain,
	using good pants constructions.
	When we build the (analogue) $1$-skeleton over the $0$-skeleton,
	it suffices, for virtual $2$-domination, to control the relative homology classes
	of the connecting geodesics.
	However, in the case of virtual $1$-domination,
	the relative homology classes	of lifts of those geodesics, which are certain paths of frames in the frame bundle,
	must also be controlled.
	Such kind of issue is exactly
	what the enhanced connection principle of \cite{Liu} takes care of,
	(see Theorem \ref{connectionprinciple} for the quoted statement).
	
	Thirdly, we shift the emphasis of our exposition
	to homological calculations of frame paths.
	Homological calculations are vital for our argument
	but quite straightforward in \cite{Sun2}.
	On the other hand, a significant part of \cite{Sun2} is devoted to demonstrating that
	the aforementioned immersed $2$-subcomplex is $\pi_1$--injective and quasiconvex.
	Its argument, and estimates, can be adapted easily to our situation.
	To avoid meaningless repeat of work,
	we only include necessary modifications in the corresponding part of this paper,
	and refer the reader to \cite[Section 4]{Sun2} for full details.
	
	The organization of the paper is as the following. In Section \ref{Sec-prelim},
	we give a quick review of good pants constructions. In Section \ref{Sec-hypCase},
	we prove the core case of main theorem and its metric refinement by assuming that
	the target is a closed oriented hyperbolic $3$-manifold. In Section \ref{Sec-generalCase},
	we prove the general case by deriving from the core case.

\subsection*{Acknowledgement}
The authors are grateful to Grant Lakeland and Christopher Leininger for informing us their work \cite{LL} on strict contraction between hyperbolic $3$-orbifolds and hyperbolic $4$-orbifolds.

\section{Preliminaries}\label{Sec-prelim}

	In this section, we give a quick review of the work of Kahn--Markovic
	and its following development.
	All the material in this section can be found in \cite{KM,LM,Liu}.

	In \cite{KM}, Kahn and Markovic proved the celebrated Surface Subgroup Theorem,
	which was employed as the first step for proving Thurston's Virtual Haken Conjecture \cite{Ag}.
	
	\begin{theorem}[Surface Subgroup Theorem]\label{surface}
		For any closed hyperbolic $3$-manifold $M$,
		there exists an immersed closed hyperbolic surface $f\colon S\looparrowright M$,
		such that $f_*\colon \pi_1(S)\rightarrow 	\pi_1(M)$ is an injective map.
	\end{theorem}
	
	The immersed subsurface of Kahn--Markovic is built
	by pasting a large collection of so-called good pants
	together in a nearly totally geodesic way. Their method has been developed
	further to build $\pi_1$--injectively immersed objects with more variations.
	In this paper, we generally use the term \emph{good pants construction}
	to refer to any construction following that strategy.
	It may also be considered as
	the package of fundamental constructions that we describe
	in this preliminary section.
	
	\subsection{Good curves and good pants}
	We fix a closed hyperbolic $3$-manifold $M$, a small number $\epsilon>0$ and a large number $R>0$.
	An $(R,\epsilon)$--\emph{good curve} is an oriented closed geodesic in $M$ with complex length $2\epsilon$--close to $R$,
	namely, its length is close to $R$ and rotation part is close to $0$.
	The set consists of all such $(R,\epsilon)$--good curves
	is denoted by ${\bold \Gamma}_{R,\epsilon}$.
	A pair of \emph{$(R,\epsilon)$-good pants}
	is a homotopy class of immersed oriented pair of pants $\Pi \looparrowright M$,
	such that the three cuffs of $\Pi$ are mapped to closed geodesics
	$\gamma_i\in {\bold \Gamma}_{R,\epsilon}$,
	and the complex half length $\bold{hl}_{\Pi}(\gamma_i)$ of $\gamma_i$ with respect to $\Pi$
	satisfies
		$$\left|\bold{hl}_{\Pi}(\gamma_i)-\frac{R}{2}\right|<\epsilon$$
	for $i=1,2,3$,
	(see \cite[Section 2.1]{KM} for the precise definition of complex length and $\bold{hl}_{\Pi}(\gamma)$).
	
	A important innovation of Kahn--Markovic is that
	good pants are pasted along a good curve with an almost $1$-shifting, rather than exactly matching seams along the common cuff.
	In complex Fenchel--Nielsen coordinates,
	this condition can be written as
		$$|s(C)-1|<\frac{\epsilon}{R},$$
	(see \cite[Section 2.1]{KM}).
	The almost $1$-shifting is a crucial condition to guarantee
	that $f_*\colon \pi_1(S)\rightarrow \pi_1(M)$ is injective.
	Moreover, Kahn and Markovic show that, for any good curve $\gamma$,
	feet of good pants are almost evenly distributed over $\gamma$.
	This roughly means that the counting measure of the feet of good pants on $\gamma$ is
	$\frac{\epsilon}{R}$--close to a scaling of
	the Lebesgue measure on (the antipodal-involution quotient of)
	unit normal bundle over $\gamma$,
	(see \cite[Theorem 3.4]{KM}).
	It can be implied that a large collection of good pants in $M$
	exist and can be pasted in the almost $1$-shifting fashion.
	Therefore, the asserted immersed surface can be constructed.

	\subsection{A mod 2 invariant of panted cobordism}
	For a closed oriented hyperbolic $3$-manifold $M$, and a point $p\in M$,
	a \emph{special orthonormal frame} (or simply a frame) of $M$ at $p$ is
	a triple of unit tangent vectors $(\vec{t}_p,\vec{n}_p,\vec{t}_p\times \vec{n}_p)$
	such that $\vec{t}_p,\vec{n}_p\in T_p^1M$ with $\vec{t}_p\perp \vec{n}_p$,
	and that $\vec{t}_p\times \vec{n}_p\in T_p^1M$ is the cross product with respect to the orientation of $M$.

	We use $\SO(M)$ to denote the frame bundle over $M$ which consists of all special orthonormal frames of $M$.
	For any oriented closed geodesic $\gamma\in {\bold \Gamma}_{R,\epsilon}$, a \emph{canonical lifting} of $\gamma$
	is defined in \cite{LM}. It is an oriented closed curve $\hat{\gamma}$ in $SO(M)$ defined as the following.

	Take a point $p\in \gamma$, and a frame ${\bold p}=(\vec{t}_p,\vec{n}_p,\vec{t}_p\times \vec{n}_p)$ over $p$
	such that $\vec{t}_p$ is tangent to $\gamma$.
	Then $\hat{\gamma}$ is defined as the concatenation of the following three paths:
	\begin{itemize}
		\item The parallel transportation of ${\bold p}$ along $\gamma$.
		Denote the resulting frame by ${\bold p'}=(\vec{t}'_p,\vec{n}'_p,\vec{t}'_p\times \vec{n}'_p)$.
		\item The closed path in $\SO(M)|_p$ which is given by a $2\pi$-rotation of ${\bold p'}$ along $\vec{t}'_p$ counterclockwise.
		Any closed path that represents the nontrivial element of $H_1(\SO(3);\mathbb{Z})$ actually works as well.
		\item The shortest path, which is $\epsilon$-short, in $\SO(M)|_p$ from ${\bold p'}$ to ${\bold p}$.
	\end{itemize}

	The homology class $[\hat{\gamma}]\in H_1(\SO(M);\mathbb{Z})$ does not depend on the choice of the point $p\in \gamma$ and the frame ${\bold p}$ at $p$.
	So a homomorphism $\Psi\colon \mathbb{Z}{\bold \Gamma}_{R,\epsilon}\to H_1(\SO(M);\mathbb{Z})$ is well defined.
	For any $L \in \mathbb{Z}{\bold \Gamma}_{R,\epsilon}$,
	the projection of $\Psi(L)$ in $H_1(M;\Integral)$
	is exactly the homology class of $L$.
	One of the main result in \cite{LM} shows that the boundary of $\mathbb{Z}{\bold \Pi}_{R,\epsilon}$,
	the free abelian group generated by the set ${\bold \Pi}_{R,\epsilon}$,
	is exactly the kernel of the homomorphism $\Psi$:
	
	\begin{theorem}\label{homology}
		Given a closed oriented hyperbolic $3$-manifold $M$,
		for small enough $\epsilon>0$ depending on $M$ and large enough $R>0$ depending on $\epsilon$ and $M$,
		denote by ${\bold \Omega}_{R,\epsilon}(M)=\Integral{\bold \Gamma}_{R,\epsilon}/\partial\Integral{\bold \Pi}_{R,\epsilon}$
		the \emph{panted cobordism group} of $M$.
		Then there is a natural isomorphism
		$$\Phi\colon{\bold \Omega}_{R,\epsilon}(M)\rightarrow H_1(\SO(M);\Integral).$$
	\end{theorem}
	
	Suppose $L\in\Integral{\bold \Gamma}_{R,\epsilon}$ is a null-homologous good multi-curve.
	Since $\Phi(L)\in H_1(\SO(M);\Integral)\cong H_1(M;\Integral)\times H_1(\SO(3);\Integral)$ has trivial projection in the first factor,
	we define $\sigma(L)$ to be the projection of $\Phi(L)$ to the second factor $H_1(\SO(3);\mathbb{Z})\cong \Integral_2$.

	The following theorem determines
	when a null-homologous good multi-curve bounds
	an immersed oriented $(R,\epsilon)$--nearly totally geodesic subsurface in $M$,
	or more precisely, an immersed subsurface pasted by $(R,\epsilon)$--good pants in the almost $1$-shifting fashion.
	This result is stated in the following form in \cite{Sun2} and proved in \cite{LM}.
	
	\begin{theorem}\label{bounding}
		Let $M$ be a closed hyperbolic $3$-manifold.
		For any small enough $\epsilon>0$ depending on $M$, and any large enough $R>0$ depending on $\epsilon$ and $M$,
		the following statement holds.
		For any null-homologous oriented $(R,\epsilon)$--multicurve $L\in \Integral{\bold \Gamma}_{R,\epsilon}$,
		there exists a nontrivial invariant $\sigma(L)\in \mathbb{Z}_2$ defined and having the following properties:
		\begin{itemize}
			\item $\sigma(L_1\sqcup L_2)=\sigma(L_1)+\sigma(L_2)$.
			\item $\sigma(L)=\bar{0}$ if and only if $L$ bounds an immersed oriented $(R,\epsilon)$--nearly totally geodesic subsurface $S$ in $M$.
			Moreover, if we associate to each component $l_i$ of $L$, $i=1,\cdots,n$,
			a normal vector $\vec{v}_i\in N^1(\sqrt{l_i})$,
			then the surface $S$ can be constructed to satisfy the following extra requirement.
			Let $C_i$ be the boundary component of $S$ which is mapped to $l_i$,
			and $\Pi_i$ be the pair of pants in $S$ having $C_i$ as one of its cuffs,
			then the foot of $\Pi_i$ on $l_i$ is $\frac{\epsilon}{R}$--close to $\vec{v}_i$.
		\end{itemize}
	\end{theorem}
	
	\subsection{An enhanced connection principle}
	The following enhanced version of connection principle is proved in \cite{Liu},
	which plays an important role in this paper.
	
	\begin{theorem}\label{connectionprinciple}
		Let $M$ be an oriented closed hyperbolic $3$-manifold.
		Let ${\bold p}= (\vec{t}_p, \vec{n}_p,\vec{t}_p\times\vec{n}_p)$
		and ${\bold q}= (\vec{t}_q, \vec{n}_q,\vec{t}_q\times\vec{n}_q)$ in $\SO(M)$
		be a pair of special orthonormal frames at points $p$ and $q$ of $M$,
		respectively.
		Let $\Xi\in H_1(\SO(M),\,{\bold p}\cup {\bold q};\,\Integral)$
		be a relative homology class with boundary $\partial_*\Xi=[{\bold q}]-[{\bold p}]$.

		Given any positive constant $\delta$, and for every sufficiently large positive constant $L$ with respect to $M$, $\Xi$, and $\delta$,
		there exists a geodesic path $s$ in $M$ from $p$ to $q$ with the following properties.
		\begin{itemize}
			\item The length of $s$ is $\frac{\delta}{L}$--close to L.
			The initial direction of $s$ is $\frac{\delta}{L}$--close to $\vec{t}_p$ and the terminal direction of $s$ is $\frac{\delta}{L}$--close to $\vec{t}_q$.
			\item The parallel transport from $p$ to $q$ along $s$ takes ${\bold p}$ to a frame ${\bold q'}$
			which is $\frac{\delta}{L}$--close to ${\bold q}$,
			and there exists a unique shortest path in $\SO(M)|_q$ between ${\bold q'}$ and ${\bold q}$.
			\item Denote by $\hat{s}$ the path which is
			the concatenation of the parallel-transport path from ${\bold p}$ to ${\bold q'}$
			with the shortest path from ${\bold q'}$ to ${\bold q}$.
			The relative homology class represented by $\hat{s}\in \pi_1(\SO(M),\,{\bold p}\cup{\bold q})$ equals $\Xi$.
		\end{itemize}
	\end{theorem}
	
	For simplicity, if the parallel transportation of ${\bold p}$ along $s$ is a frame ${\bold q}'$
	which is very close ($\frac{\delta}{L}$--close) to a frame ${\bold q}$ based at the same point as ${\bold q}'$,
	we introduce the notation
		$$[{\bold p}\xrightarrow{s} {\bold q}]$$
	to denote the relative homology class in $H_1(\SO(M),{\bold p}\cup{\bold q};\Integral)$
	represented by the concatenation of the parallel-transport path from ${\bold p}$ to ${\bold q'}$ along $s$
	with the shortest path in $SO(M)|_q$ from ${\bold q'}$ to ${\bold q}$.
	
	\subsection{Finer estimates}\label{epsilonbyL}
		As applied to \cite{Sun1,Sun2,Liu}, (see \cite[Remark D]{Sun2} for example,)
		the $\epsilon$ (or $\delta$) of Theorem \ref{homology} and Theorem \ref{bounding} can
		be replaced by $\frac{\epsilon}{R}$ (or $\frac{\delta}L$),
		since the exponential mixing rate of
		the frame flow beats any polynomial rate.
		In this paper, we also apply the refined version of these theorems
		so as to invoke established estimates from \cite{Sun1,Sun2,Liu}.

\section{Virtual 1-domination between hyperbolic 3-manifolds}\label{Sec-hypCase}
	To emphasize the geometric information that we extract from the target manifold,
	we prove in this section
	the following core case of Theorems \ref{main} and \ref{contraction},
	leaving the derivation of the general case to next section.

	\begin{theorem}\label{virtual_one_domination_hyperbolic}
		For any pair of closed oriented hyperbolic $3$-manifolds $M$ and $N$, there exists a finite cover $M'$ of $M$ and a degree-$1$	map $f\colon M'\rightarrow N$.
		When a positive constant $\kappa$ is given, it can be required in addition that $f$ is $\kappa$-Lipschitz.
	\end{theorem}

	To sketch the idea, let us fix a geometric triangulation of $N$, namely, a realization of $N$
	by gluing finitely many (ordinary, geodesic) hyperbolic tetrahedra.
	
	We want to construct a $\pi_1$--injectively immersed quasiconvex $2$-subcomplex $Z\looparrowright M$
	which homologically looks like the $2$-skeleton of $N$.
	To be more precise, we want $Z$ to be homeomorphic to the $2$-skeleton of $N$
	but	with the $2$-simplices replaced by $1$-holed surfaces.
	To immerse the analogue $1$-skeleton $Z^1$ of $Z$, we may first embed the vertices into $M$,
	and then use the connection principle to immerse the edges
	to be long geodesics between the embedded vertices.
	However, in this way,
	the boundary cycles of $2$-simplices may fail to be homologically trivial in $M$;
	even if they are, the closed geodesics of their free-homotopy classes
	may fail to bound immersed panted subsurfaces,
	as the $\Integral_2$--valued obstructions may be nontrivial.
	So we would not be able to proceed by good pants constructions to obtain a desired $2$-complex $Z$.
	To design the immersion of $Z^1$ more carefully,
	we need to invoke the enhanced connection principle to control the relative homology classes
	of the constructed long geodesics and (certain) frame paths associated to them.
	Correct choices of those relative homology classes comes from a homomorphism of the form
	$H_1(\SO(N),F_N;\Integral)\to H_1(\SO(M),F_M;\Integral)$, where $F_N$ and $F_M$
	are certain finite collection of frames over the $0$-skeleton.
	This homomorphism is referred to as our homological instruction (Proposition \ref{instruction}).
	The instruction about the local shape of $Z^1$ near its $0$-skeleton
	comes from the shape of $1$-skeleton of $N$ near its vertices.
	More precisely, angles between edges at vertices of $Z^1$,
	are approximately the corresponding angles in $N$,
	and parallel transport of frames along edges
	are approximately the same as that of $N$ accordingly.
	The geometric design facilitates our calculation,
	so we can check that the $\Integral_2$--valued obstructions indeed vanish for
	the expected boundary cycles in $Z^1$.
	Eventually, we are able to extend $Z^1$
	to be the claimed $Z$ using good pants constructions.
	
	After the construction of $Z$, the virtual $1$-domination can be constructed in a similar way as done by \cite{Sun2}.
	In fact, the geometric design also guarantees the quasiconvexity of $Z$ and the shape of its convex core.
	We pass to a finite cover $M'$ of $M$ into which $Z$ lifts to be an embedding, by the LERF property of $\pi_1(M)$.
	The claimed $1$-domination $M'\to N$ can be naturally defined restricted to $Z$,
	and can be extended over $M'$ almost automatically.
	
	The rest of this section is devoted to the proof of Theorem \ref{virtual_one_domination_hyperbolic}.

	\subsection{Setup}
	Suppose that $M$ and $N$ are a pair of closed oriented hyperbolic $3$-manifolds.
	We denote the special orthonormal frame bundle of $M$ and $N$ as $\SO(M)$ and $\SO(N)$ accordingly.
	
	\subsubsection{Intitial data}
	Modulo good pants constructions and lifting to finite covers,
	our construction depends (only) on the following choice of auxiliary data:
	\begin{itemize}
		\item A geometric triangulation of $N$:
		That is, a division of the target hyperbolic 3-manifold $N$
		into a simplicial complex of which the simplices are all totally geodesic.
		When a positive constant $\kappa$ is given,
		we require that the length of edges are at most $10^{-2}\kappa$,
		or at most $10^{-2}$ if $\kappa$ exceeds $1$.
		\item A homological lift, by which we mean a homomorphism of homology:
			$$i_{*}\colon H_1(N;\Integral)\to H_1(M;\Integral),$$
		together with an embedding of the $1$-skeleton:
			$$i_1\colon N^{(1)}\to M.$$
		We require $i_1$ to realize $i_*$ in the sense that
		the induced homomorphism $i_{1*}\colon H_1(N^{(1)};\Integral)\to H_1(M;\Integral)$
		factors through $i_*$.		
		\item A pair of trivializations of the $\SO(3)$--principal bundles:
			$$t_N\colon \SO(N)\to N\times \SO(3)$$
		and
			$$t_M\colon \SO(M)\to M\times \SO(3).$$
	\end{itemize}

	A collection of initial data can be easily prepared as required.
	Specifically, a geometric triangulation can be constructed by the following way:
	First take a Dirichlet polyhedron $P$ in the universal cover $\mathbb{H}^3$ of $N$.
	The combinatorial barycentric subdivision of $P$ can be realized
	as a geometric triangulation of $P$, and the barycenters can be chosen to match up under side-pairing.
	Then we have an induced geometric triangulation of $N$ as claimed.
	By iterating barycentric subdivisions, we can bound the edge length by any required
	positive constant $10^{-2}\kappa$.
	A default option for a homological lift could be
	the trivial homomorphism together with an embedding of the $1$-skeleton of $N$
	into a topologically embedded ball of $M$.
	Choosing trivializations of the frame bundles
	amounts to finding a field of frames
	over each base manifold, which is always possible as the tangent bundle
	of any orientable closed $3$-manifold is trivial.

	\subsubsection{Geometric notations}
	Having chosen the geometric triangulation of $N$,
	we enumerate the vertices of $N$ by
		$$V_N=\{n_1,n_2,\cdots,n_l\}.$$
	When there is an edge between $n_i$ and $n_j$,
	denote the oriented edge from $n_i$ to $n_j$ by $e_{ij}$,
	so $e_{ji}$ is the orientation-reversal of the same edge.
	When there is a $2$-simplex spanned by a triple of vertices $(n_i, n_j, n_k)$,
	denote the marked $2$-simplex by $\Delta_{ijk}$.
	
	Associated to any marked $2$-simplex $\Delta_{ijk}$, there is a frame
	${\bold F}_{ijk}$ based at $n_i$ defined as follows.
	For any oriented edge $e_{ij}$,
	denote by $\vec{v}_{ij}$ the unit tangent vector of $e_{ij}$ based at $n_i$ pointing to $n_j$.
	For any marked $2$-simplex $\Delta_{ijk}$, denote by
		$$\vec{n}_{ijk}=\frac{\vec{v}_{ij}\times \vec{v}_{ik}}{|\vec{v}_{ij}\times \vec{v}_{ik}|}$$
	the distinguished normal vector of $\Delta_{ijk}$ at $n_i$.
	The cross product does not vanish since $\Delta_{ijk}$ is totally geodesic in $N$.
	Then the frame associated to $\Delta_{ijk}$ based at $n_i$ is defined by
		$${\bold F}_{ijk}=(\vec{v}_{ij},\vec{n}_{ijk}).$$
	In this paper,
	we only write the first two vectors for a frame, such as
	${\bold F}=(\vec{v},\vec{n})$,
	implicitly meaning that the third vector is $\vec{v}\times \vec{n}$.
	We also introduce the notation
		$$-{\bold F}_{ijk}=(-\vec{v}_{ij},-\vec{n}_{ijk}).$$
	Note that if we write down all three vectors components in a frame,
	$-{\bold F}_{ijk}$ only changes sign for the first two vectors from that of ${\bold F}_{ijk}$,
	not for the third.
	Observe that $-{\bold F}_{jik}$ is the parallel transport of ${\bold F}_{ijk}$
	from $n_i$ to $n_j$ along $e_{ij}$.	
	Collectively we write
	$$F_N=\{\pm{\bold F}_{ijk}\colon \Delta_{ijk}\textrm{ a marked 2-simplex of }N\}
	\,\subset\,\SO(N)|_{V_N}.$$
	%
		
	With the chosen homological lift and the trivialization of frame bundles, we define a morphism
	of $\SO(3)$--principal bundles:	
		$${\bold i}_1\colon \SO(N)|_{N^{(1)}}\to \SO(M)$$
	over a base map	$i_1\colon N^{(1)}\to M$
	by the following commutative diagram:
	\begin{equation}\label{define_bold_i_1}
		\begin{CD}
			\SO(N)|_{N^{(1)}} @> {{\bold i}_1}>> \SO(M)\\
			@V t_N VV @VV t_M V\\
			N^{(1)}\times \SO(3) @>i_1\times \mathrm{id}>> M\times \SO(3).
		\end{CD}
	\end{equation}
		
	Via the bundle morphism ${\bold i}_1$, we define objects associated to $M$ accordingly, namely,
		$$m_k=i_1(n_k)$$
	and
		$${\bold F}_{ijk}^M=(\vec{v}_{ij}^M,\vec{n}_{ijk}^M)=\mathbf{i}_1({\bold F}_{ijk}).$$
	Collectively we write	
		$$V_M=\{m_1,m_2,\cdots,m_l\}\subset M,$$
	and
		$$F_M=\mathbf{i}_1(F_N)\,\subset\,\SO(M)|_{V_M}.$$
	Observe that since $\mathbf{i}_1$ commutes with the right $\SO(3)$--multiplication,
	$-{\bold F}_{ijk}^M$ equals $\mathbf{i}_1(-{\bold F}_{ijk})$.
	
	Denote by $\theta_{ijk}$ the external angle of the geodesic $2$-simplex $\Delta_{ijk}$ in $N$
	at the vertex $n_i$. In other words, it is the geometric angle formed by
	the vectors $-\vec{v}_{ij}$ and $\vec{v}_{ik}$, valued in the interval $(0,\pi)$.
	Because there are only finitely many $2$-simplices $\Delta_{ijk}$,
	there is a uniform bound constant
		$$\phi_0=\min\{\theta_{ijk},\pi-\theta_{ijk}\colon \Delta_{ijk}\textrm{ any 2-simplex of }N\}\in(0,\pi)$$
	such that
		$$\theta_{ijk}\in[\phi_0,\pi-\phi_0].$$	
	
	\subsubsection{Homological instruction}
	Our construction follows the homological instruction given by the homomorphism
	declared by the following proposition.		
	
	\begin{proposition}\label{instruction}
		There exists a homomorphism:
			$${\bold j}_*\colon H_1(\SO(N),F_N;\,\Integral)\to H_1(\SO(M),F_M;\,\Integral),$$
		which is uniquely induced by the bundle morphism ${\bold i}_1$.
		Moreover, the following diagrams of homomorphisms are commutative:
		\begin{equation}\label{endpoint}
			\begin{CD}
			H_1(\SO(N),F_N;\,\Integral) @>\partial_*>> H_0(F_N;\,\Integral)\\
			@V{\bold j}_*VV @VV{\bold i}_{1*}V\\
			H_1(\SO(M),F_M;\,\Integral) @>\partial_*>> H_0(F_M;\,\Integral),
			\end{CD}
		\end{equation}
		and
		\begin{equation}\label{subframe}
			\begin{CD}
			H_1(\SO(N)|_{V_N},F_N;\,\Integral) @>\mathrm{incl}_*>> H_1(\SO(N),F_N;\,\Integral)\\
			@V{\bold i}_{1*}VV @VV{\bold j}_*V\\
			H_1(\SO(M)|_{V_M},F_M;\,\Integral) @>\mathrm{incl}_*>> H_1(\SO(M),F_M;\,\Integral).
			\end{CD}
		\end{equation}
		The horizontal homomorphisms are those of the long exact sequences for pairs
		and the vertical homomorphisms involving ${\bold i}_1$ are induced by restrictions.
	\end{proposition}
	
	\begin{proof}
		By definition, there is an induced homomorphism
			$${\bold i}_{1*}\colon H_1(\SO(N)|_{N^{(1)}},F_N;\,\Integral)\to H_1(\SO(M),F_M;\,\Integral).$$		
		The relative homology $H_2(\SO(N),\SO(N)|_{N^{(1)}};\,\Integral)$
		is generated by the $2$-simplices of $N$,
		or more precisely,
		by the $t_N$--horizontal $2$-simplices of $\SO(N)$.
		The assumption that $i_1\colon N^{(1)}\to M$ induces a homomorphism $i_*\colon H_1(N;\Integral)\to H_1(M;\Integral)$
		forces boundaries of $t_N$--horizontal $2$-simplices of $\SO(N)$
		to be sent via ${\bold i}_1$ to	$t_M$--horizontal boundaries in $\SO(M)$,	
		by the diagram (\ref{define_bold_i_1}).
		Since $F_M$ is $0$-dimensional,
		the restriction of ${\bold i}_{1*}$ to $\partial_* H_2(\SO(N),\SO(N)|_{N^{(1)}};\,\Integral)$ is hence trivial.
		Moreover, the inclusion induces an epimorphism $H_1(\SO(N)|_{N^{(1)}},F_N;\,\Integral)\to H_1(\SO(N),F_N;\,\Integral)$.
		It follows from the long exact sequence for the triple $(\SO(N),\SO(N)|_{N^{(1)}},F_N)$ that
		${\bold i}_{1*}$ descends to a unique homomorphism ${\bold j}_*$ as claimed.
		The commutative diagrams are immediate consequences of the naturality of the long exact sequence,
		for the map of pairs ${\bold i}_1\colon (\SO(N)|_{N^{(1)}},F_N)\to (\SO(M),F_M)$,
		and the way ${\bold j}_*$ is induced.		
	\end{proof}

	\subsection{Immersing a quasiconvex 2-complex}
	Provided with the initial data and the homological instruction,
	we construct a $\pi_1$--injectively immersed quasiconvex $2$-complex
	into the domain:
		$$Z\looparrowright M,$$
	by invoking good pants constructions.
	Assume that some small $\delta>0$ and large $L>0$ have been appropriately selected.
	These parameters are supposed to guarantee the geometry of resulting complex $Z$
	as we want, and their ranges are to be determined at the end of this subsection.
	
	The abstract $2$-complex $Z$ can be topologically decomposed into vertices, segments,
	and pieces that are homeomorphic to compact orientable surfaces with connected boundary.
	Those parts are constructed in natural correspondence with the 0-, 1-, and 2-simplices of $N$.
	In particular, the analogue $0$-skeleton, denoted by $Z^0$,
	is a discrete finite set of vertices enumerated by the same index set as $V_N$.
	So there is a naturally induced embedding of $Z^0$ into $M$ as $V_M$, denoted by:
		$$Z^0\hookrightarrow M.$$
	
	\subsubsection{Wiring in $M$ with the $1$-skeleton of $N$}	
	Corresponding to each edge $e_{ij}$ from $n_i$ to $n_j$ of $N$, we create an oriented geodesic arc $e_{ij}^M$ from $m_i$ to $m_j$
	invoking the enhanced connection principle.
	To avoid ambiguity, suppose that $i<j$, and take $k$ to be the smallest index such that $\Delta_{ijk}$ is a registered $2$-simplex of $N$.
	We only construct $e_{ij}^M$ with respect to such triples of vertices $(n_i,n_j,n_k)$, as follows.
	
	Because of the diagram (\ref{endpoint}), the relative homology class
		$${\bold j}_*[{\bold F}_{ijk}\xrightarrow{e_{ij}}(-{\bold F}_{jik})]\in H_1(\SO(M),{\bold F}_{ijk}^M\cup(-{\bold F}_{jik}^M);\,\Integral)$$
	has boundary $[-{\bold F}_{jik}^M]-[{\bold F}_{ijk}^M]$ in $H_0({\bold F}_{ijk}^M\cup(-{\bold F}_{jik}^M);\,\Integral)$,
	which can be canonically identified as a submodule of $H_0(F_M;\,\Integral)$.
	By the enhanced connection principle (Theorem \ref{connectionprinciple}),
	we construct an oriented geodesic arc $e_{ij}^{M}$ in $M$ from $m_i$ to $m_j$
	with the following description:
	\begin{itemize}
	\item The length of $e_{ij}^M$ is $\frac{\delta}{L}$--close to $L$,
	the initial direction of $e_{ij}^M$ is $\frac{\delta}{L}$--close to the first vector of ${\bold F}_{ijk}^M$
	and the terminal direction of $e_{ij}^M$ is $\frac{\delta}{L}$--close to the first vector of $-{\bold F}_{jik}^M$.
	\item The parallel transport of ${\bold F}_{ijk}^M$ along $e_{ij}^M$ is $\frac{\delta}{L}$--close to $-{\bold F}_{jik}^M$,
	and there is a unique shortest path in $\SO(M)|_{m_j}$ between these two frames.
	\item In $H_1(\SO(M),{\bold F}_{ijk}^M\cup(-{\bold F}_{jik}^M);\,\Integral)$,
	$$[{\bold F}_{ijk}^M\xrightarrow{e_{ij}^M}(-{\bold F}_{jik}^M)]={\bold j}_*[{\bold F}_{ijk}\xrightarrow{e_{ij}}(-{\bold F}_{jik})].$$	
	\end{itemize}
	Construct $e_{ij}^M$ accordingly for all oriented edges $e_{ij}$ pairs with $i<j$
	with respect to the smallest possible $k$ as above.
	We define $e_{ji}^M$ to be the orientation-reversal of $e_{ij}^M$.	
	
	One could certainly construct an oriented geodesic arc $e_{ij}^M$ with respect to a different available $k$, or
	construct $e_{ij}^M$ for $i>j$ directly, satisfying the same description.
	However, such extra work is unnecessary, nor much helpful.
	The point is that what we have constructed works for all:
	
	\begin{lemma}\label{welldefine}
		For any triple of vertices $(n_i,n_j,n_k)$ spanning a marked $2$-simplex $\Delta_{ijk}$ of $N$, the corresponding
		oriented geodesic arc $e_{ij}^M$ of our construction also satisfies the description with respect to $(n_i,n_j,n_k)$.
	\end{lemma}
	
	\begin{proof}
		It is easy to see that the first two items in the description remains true if we take
		any other $2$-simplex $\Delta_{ijk'}$ or consider $e_{ji}$,
		by the fact that ${\bold i}_1$ commutes with the right $\SO(3)$--action.
		So our goal is to verify the third item of the description,
		and it amounts to prove the following statement:
		For an edge $e_{ij}$ in the triangulation of $N$ from $n_i$ to $n_j$,
		and a geodesic arc $e_{ij}^M$ in $M$ from $m_i$ to $m_j$,
		Suppose that the following equality of relative homology classes holds for some marked $2$-simplex $\Delta_{ijk}$ of $N$:
		\begin{equation}\label{F_ijk}
			[{\bold F}_{ijk}^M\xrightarrow{e_{ij}^M}(-{\bold F}_{jik}^M)]={\bold j}_*[{\bold F}_{ijk}\xrightarrow{e_{ij}}(-{\bold F}_{jik})],
		\end{equation}
		then for any other marked $2$-simplex $\Delta_{ijk'}$ of $N$,
		\begin{equation}\label{F_ijk_prime}
			[{\bold F}_{ijk'}^M\xrightarrow{e_{ij}^M}(-{\bold F}_{jik'}^M)]={\bold j}_*[{\bold F}_{ijk'}\xrightarrow{e_{ij}}(-{\bold F}_{jik'})],
		\end{equation}
		and
		\begin{equation}\label{F_jik}
			[{\bold F}_{jik}^M\xrightarrow{e_{ji}^M}(-{\bold F}_{ijk}^M)]={\bold j}_*[{\bold F}_{jik}\xrightarrow{e_{ji}}(-{\bold F}_{ijk})].
		\end{equation}

		We first verify the identity (\ref{F_ijk_prime}). Denote by
			$$A(\theta)=\begin{bmatrix}
			1 & 0 & 0\\
			0 & \cos{\theta} & \sin{\theta}\\
			0 & -\sin{\theta} & \cos{\theta}
			\end{bmatrix},$$
		the $3\times 3$ orthogornal matrix which represents a rotation about the first vector.
		Since the first vectors in both ${\bold F}_{ijk}$ and ${\bold F}_{ijk'}$ are $\vec{v}_{ij}$,
		we have $${\bold F}_{ijk'}={\bold F}_{ijk}\cdot A(\theta_0).$$
		Here $\theta_0$ is the dihedral angle between $\Delta_{ijk}$ and $\Delta_{ijk'}$.
		Since $2$-simplices are totally geodesic in $N$, we have
			$$-{\bold F}_{jik'}=(-{\bold F}_{jik})\cdot A(\theta_0).$$

		Similar to the notation of relative homology classes given by parallel transportations,
		we use
			$$[{\bold F}_{ijk}\xrightarrow{A(\theta_0t)}{\bold F}_{ijk'}]$$
		to denote the relative homology class in $H_1(\SO(N),F_N;\,\Integral)$
		which is given by the path
		$t\mapsto {\bold F}_{ijk}\cdot A(\theta_0t)$
		with the initial point ${\bold F}_{ijk}$ and the terminal point ${\bold F}_{ijk'}$.

		Note that the parallel transportation along a fixed geodesic commutes with
		the right $\SO(3)$-action on the frame bundle.
		By considering the parallel transportation of ${\bold F}_{ijk}\cdot A(\theta_0t)$ along $e_{ij}$ for all $t\in [0,1]$,
		we get a map $I\times I\to \SO(N)$ such that the relative homology class of its boundary in $H_1(\SO(N),F_N;\,\Integral)$ is
		\begin{equation}\label{zeroinN}
		[{\bold F}_{ijk},-{\bold F}_{jik}]+[-{\bold F}_{jik},-{\bold F}_{jik'}]
		-[{\bold F}_{ijk'},-{\bold F}_{jik'}]-[{\bold F}_{ijk},{\bold F}_{ijk'}]=0.
		\end{equation}
		Here we have suppressed the understood paths to abbreviate the expression.

		Passing to $M$, we again observe that the identification ${\bold i}_1\colon \SO(N)|_{V_N}\to \SO(M)|_{V_M}$
		commutes with the right $\SO(3)$--action. Therefore,
			$${\bold F}_{ijk'}^M={\bold F}_{ijk}^M\cdot A(\theta_0) \ \text{and}\ -{\bold F}_{jik'}^M=(-{\bold F}_{jik}^M)\cdot A(\theta_0).$$

		Let $-\bar{{\bold F}}_{jik}^M$ be the parallel transportation of ${\bold F}_{ijk}^M$ along $e_{ij}^M$.
		Note that $-\bar{{\bold F}}_{jik}^M$ is not exactly $-{\bold F}_{jik}^M$,
		but there exists a $\frac{\delta}{L}$--short geodesic $\gamma\colon I\to \SO(3)$ with $\gamma(0)=I_3$
		such that $-{\bold F}_{jik}^M=-\bar{{\bold F}}_{jik}^M\cdot \gamma(1)$.
		We can define a map $I\times I\to \SO(M)$ as follows.
		For every $t\in [0,1]$,
		the $t$--slice is the concatenation of a path defined by
		the parallel transportation of ${\bold F}_{ijk}^M\cdot A(\theta_0t)$ along $e_{ij}^M$
		(ending as $(-\bar{{\bold F}}_{jik}^M)\cdot A(\theta_0t)$),
		and a path defined by the $\SO(3)$--action of $A(\theta_0t)^{-1}\cdot \gamma \cdot A(\theta_0t)$,
		(ending as $(-{\bold F}_{jik}^M) \cdot A(\theta_0t)$).

		The relative homology class of the boundary of this mapped-in square, in $H_1(\SO(M),F_M;\,\Integral)$, is
		\begin{equation}\label{zeroinM}
		[{\bold F}_{ijk}^M,-{\bold F}_{jik}^M]+[-{\bold F}_{jik}^M,-{\bold F}_{jik'}^M]-[{\bold F}_{ijk'}^M,-{\bold F}_{jik'}^M]-[{\bold F}_{ijk}^M,{\bold F}_{ijk'}^M]=0.
		\end{equation}
		The suppressed path labels are $e_{ij}^M$, $A(\theta_0t)$, $e_{ij}^M$, and $A(\theta_0t)$ accordingly.
				
		By the diagram (\ref{subframe}) and the fact that ${\bold i}_1\colon SO(N)|_{V_N}\to SO(M)|_{V_M}$ commutes with the right $\SO(3)$--action,
		we have
		\begin{equation}\label{correspondence}
		\begin{split}
		& [(-{\bold F}_{jik}^M)\xrightarrow{A(\theta_0t)}(-{\bold F}_{jik'}^M)]={\bold j}_*[(-{\bold F}_{jik})\xrightarrow{A(\theta_0t)}(-{\bold F}_{jik'})],\\
		& [{\bold F}_{ijk}^M\xrightarrow{A(\theta_0t)}{\bold F}_{ijk'}^M]={\bold j}_*[{\bold F}_{ijk}\xrightarrow{A(\theta_0t)}{\bold F}_{ijk'}].
		\end{split}
		\end{equation}

		As ${\bold j}_*$ is a homomorphism,
		by the equalities (\ref{zeroinN}), (\ref{zeroinM}), (\ref{correspondence}), and the assumed identity (\ref{F_ijk}),
		we obtain
			$$[{\bold F}_{ijk'}^M\xrightarrow{e_{ij}^M}(-{\bold F}_{jik'}^M)]={\bold j}_*[{\bold F}_{ijk'}\xrightarrow{e_{ij}}(-{\bold F}_{jik'})].$$
		This verifies the claimed identity (\ref{F_ijk_prime}).
				
		The verification of the identity (\ref{F_jik}) is similar, so we only give a sketch here.
		
		Denote by
		$$C(\theta)=\begin{bmatrix}
		\cos{\theta} & \sin{\theta} & 0\\
		-\sin{\theta} & \cos{\theta} & 0\\
		0 & 0 & 1
		\end{bmatrix}$$
		the $3\times 3$ orthogornal matrix which represents a rotation about the third vector.
		By considering the parallel transportation of ${\bold F}_{ijk}\cdot C(\pi t)$ along $e_{ij}$,
		we get a map $I\times I\to \SO(N)$, of which the relative homology class of the boundary in $H_1(\SO(N),F_N;\,\Integral)$ is
		\begin{equation}\label{zeroinN2}
		[{\bold F}_{ijk},-{\bold F}_{jik}]+[-{\bold F}_{jik},{\bold F}_{jik}]-[-{\bold F}_{ijk},{\bold F}_{jik}]-[{\bold F}_{ijk},-{\bold F}_{ijk}]=0.
		\end{equation}
		The suppressed path labels are $e_{ij}$, $C(\pi t)$, $e_{ij}$, and $C(\pi t)$ accordingly.
		Note that
			$$-[(-{\bold F}_{ijk})\xrightarrow{e_{ij}}{\bold F}_{jik}]=[{\bold F}_{jik}\xrightarrow{e_{ji}}(-{\bold F}_{ijk})]$$
		holds.

		Likewise to the above situation, there is a vanishing relative homology class in $H_1(\SO(M),F_M;\mathbb{Z})$,
		which is the boundary of a likewise mapped-in square in $\SO(M)$.
		We proceed in a similar fashion as before to complete the verification,
		with the only difference explained as the following.
		
		While the paths in $\SO(N)$ giving $[(-{\bold F}_{ijk})\xrightarrow{e_{ij}}{\bold F}_{jik}]$
		and $[{\bold F}_{jik}\xrightarrow{e_{ji}}(-{\bold F}_{jik})]$ are exactly the inverse of each other,
		the paths in $\SO(M)$ giving $[(-{\bold F}_{ijk}^M)\xrightarrow{e_{ij}^M}{\bold F}_{jik}^M]$
		and $[{\bold F}_{jik}^M\xrightarrow{e_{ji}^M}(-{\bold F}_{jik}^M)]$
		may differ by a uniform short distance along the parallel transportation part.
		However, this does not affect the fact that
		the relative homology classes in $H_1(\SO(M),F_M;\mathbb{Z})$ are the negative of each other,
		which can again be shown by a mapped-in-square argument.
		So by comparing the equalities and using the definition that $e_{ij}^M$ is the orientation-reversal of $e_{ji}^M$, one will obtain
			$$-[(-{\bold F}_{ijk}^M)\xrightarrow{e_{ij}^M}{\bold F}_{jik}^M]=[{\bold F}_{jik}^M\xrightarrow{e_{ji}^M}(-{\bold F}_{ijk}^M)],$$
		which verifies the identity (\ref{F_jik}).
	\end{proof}
	
	We attach edges to the analogue $0$-skeleton $Z^0$ to build the analogue $1$-skeleton $Z^1$,
	so that $Z^1$ is naturally isomorphic to the $1$-skeleton of $N$. The geodesic arcs $e_{ij}^M$
	that we have constructed give rise to a distinguished immersion:
		$$Z^1\looparrowright M.$$
	
	\subsubsection{Killing triangles with good surfaces}
	Corresponding to each $2$-simplex $\Delta_{ijk}$ spanned by a triple of vertices $(n_i,n_j,n_k)$
	of $N$, we create an immersed nearly totally geodesic subsurface of $M$
	with connected cornered boundary, $S_{ijk}^M$, such that $\partial S_{ijk}^M$ is the cyclic concatenation of the segments $e_{ij}^M$, $e_{jk}^M$,
	and $e_{ki}^M$. Combinatorially $S_{ijk}^M$ is built as
	a panted surface collared by a three-spike crown
	(an annulus with three corners on the outer boundary).
	Since these immersed subsurfaces of $M$ should correspond to the $2$-simplices of $N$
	regardless of the marking, we only need to construct $S_{ijk}^M$ for those marked $\Delta_{ijk}$
	such that $i<j<k$, as follows.
	
	Denote by $\gamma_{ijk}$ the unique oriented closed geodesic of $M$
	which is freely homotopic to the cyclic concatenation of
	the three oriented geodesic arcs $e_{ij}^M$, $e_{jk}^M$ and $e_{ki}^M$.
	The following lemma describes the shape and framing homology class of $\gamma_{ijk}$.
	
	We denote by
		$$I(\theta)=2\log(\sec(\theta/2))$$
	the limit inefficiency associated to a bending angle $\theta\in(0,\pi)$,
	(see \cite[Section 4.2]{LM} for the geometric meaning of this quantity).
	Recall that we have introduced a uniform bound $\phi_0\in(0,\pi)$
	depending only on the chosen geometric triangulation
	of $N$, so that all the $2$-simplex external angles	$\theta_{ijk}$ of $N$
	lie in the interval $[\phi_0,\pi-\phi_0]$.
	Recall also the $\Integral_2$-valued panted-cobordism invariant $\sigma$ from Theorem \ref{bounding}.

	\begin{lemma}\label{noObstruction}
		For any triple of vertices $(n_i,n_j,n_k)$ spanning a marked $2$-simplex $\Delta_{ijk}$ of $N$, the cyclic concatenation
		of the corresponding oriented geodesic arcs $e_{ij}^M$, $e_{jk}^M$, $e_{ki}^M$
		is homotopic to a unique oriented closed geodesic $\gamma_{ijk}$ of $M$. Moreover,
		the following properties hold true:
		\begin{enumerate}
		\item
			There exists an immersed annulus cobounded by
			the cyclic concatenation of $e_{ij}^M$, $e_{jk}^M$, and $e_{ki}^M$ and
			the curve $\gamma_{ijk}$. It is nearly totally geodesic and untwisted,
			as can be described by the following:
			\begin{itemize}
			\item The annulus has a triangulation with vertices $m_i$, $m_j$, $m_k$
			and three other vertices on $\gamma_{ijk}$.
			\item The immersion is totally geodesic on each $2$-simplex and has bending angles along common edges
			bounded by $\frac{10\delta}{L}$;
			\item The normal vectors along $e_{ij}^M$ of the $2$-simplex containing $e_{ij}^M$ are $\frac{10\delta}L$--close
			to the $\vec{n}$-vectors of the parallel transports of ${\bold F}_{ijk}^M$ accordingly,
			and similar statements hold for	$e_{jk}^M$ and $e_{ki}^M$.
			\end{itemize}
		\item
			The curve $\gamma_{ijk}$ is good and panted null-cobordant,
			as implied by the following description:
			\begin{itemize}
				\item The curve $\gamma_{ijk}$ is $(R_{ijk},\frac{100\delta}{L})$--good
				where $R_{ijk}=3L-I(\theta_{ijk})-I(\theta_{jki})-I(\theta_{kij})$,
				uniformly contained in $[3L-3I(\pi-\phi_0),3L-3I(\phi_0)]$.
				\item The curve $\gamma_{ijk}$ is null-homologous in $M$.
				\item The invariant $\sigma(\gamma_{ijk})=0\bmod 2$.
			\end{itemize}
		\end{enumerate}
	\end{lemma}
	
	\begin{proof}
		The first item of Part (2), about the geometry of $\gamma_{ijk}$,
		follows directly from the formulas of \cite[Lemma 4.8]{LM}.
		Note that we have been working under the hypothesis that
		the parameters $\delta$ and $L$ are suitably given.
		In particular, we see that $\gamma_{ijk}$ is homotopically nontrivial.
		
		For Part (1) about the geometry of the immersed annulus,
		the situation is very similar to \cite[Lemma 3.5]{Sun2} and the discussion thereafter,
		so we give a very brief sketch. The free homotopy between
		the cyclic concatenation of $e_{ij}^M$, $e_{jk}^M$ and $e_{ki}^M$ and $\gamma_{ijk}$
		can be realized by a map of an annulus.
		To be specific, we may triangulate the annulus by adding three vertices $p_i$, $p_j$ and $p_k$
		on $\gamma$, and adding six edges, say, $p_im_i$, $p_im_j$,
		$p_jm_j$, $p_jm_k$, $p_km_k$, and $p_km_i$.
		The mapped-in annulus can be homotoped relative to boundary so that the triangulation is
		totally geodesic on each simplex, and the edges $p_im_i$, $p_jm_j$, and $p_km_k$ are perpendicular
		to $\gamma_{ijk}$.
		Then all the claimed estimates follow directly from the formulas of \cite[Lemma 4.8]{LM}.
		
		It remains to show the homological properties about $\gamma_{ijk}$,
		namely, the last two items of Part (2).
		The homology class of $\gamma_{ijk}$ can be seen
		easily from the construction of $Z^1\looparrowright M$.
		In fact, by projecting paths of frames in the base manifold, it is clear that the cycle
		$e_{ij}^M+e_{ij}^M+e_{ij}^M$ represents the class $i_*[\partial\Delta_{ijk}]$
		in $H_1(M;\Integral)$, where the latter equals zero.
		
		To prove the vanishing of $\sigma(\gamma_{ijk})$.
		we need to explain the constructive definition of $\sigma$ in more details.
		For any $(R,\epsilon)$--good curve $\gamma$ of $M$,	recall that a \emph{canonical lifting} of $\gamma$ is
		a loop of frames in $\SO(M)$ defined as the following:
		Take a point $p\in \gamma$ and a frame ${\bold F}=(\vec{v}_p,\vec{n}_p,\vec{v}_p\times \vec{n}_p)$ based at $p$
		such that $\vec{v}_p$ is tangent to $\gamma$;
		starting from ${\bold F}$, draw a path of frames by parallel trasporting
		${\bold F}$ along $\gamma$ for one round, arriving at a frame ${\bold F}'$ over the same point $p$;
		continue by making a full rotation of ${\bold F'}$ (counterclockwise) about its first vector;
		close the path by joining ${\bold F}'$ to ${\bold F}$ by the the shortest path in $\SO(M)|_p$.
		When $\gamma$ is null-homologous in $H_1(M;\Integral)$, any canonical lifting of $\gamma$
		defines one and the same class
		in the fiber submodule $H_1(\SO(3);\Integral)\cong \Integral_2$
		of $H_1(\SO(M);\Integral)$.
		In this case, the invariant $\sigma(\gamma)$ can be
		explicitly identified with the class of the canonical lifting of $\gamma$.
		
		In the following, we denote the unique nontrivial element of the fiber submodule
		$H_1(\SO(3);\Integral)$ by $c$. The vanishing of $\sigma(\gamma)$ is equivalent to the identity:
		\begin{equation}\label{identity_c}
			[{\bold F}\xrightarrow{\gamma_{ijk}}{\bold F}]=c,
		\end{equation}
		for some (hence any) frame ${\bold F}$ at a point $p\in\gamma_{ijk}$ whose first vector is tangent to $\gamma_{ijk}$.
		The label of the arrow is considered as the closed path with both endpoints the point $p$.
				
		To prove the identity (\ref{identity_c}), we start by working with geodesic $2$-simplices of $N$. Denote by
			$$B(\theta)=\begin{bmatrix}
			\cos{\theta}  & 0 & \sin{\theta}\\
			0 & 1 & 0\\
			-\sin{\theta} & 0 & \cos{\theta}
			\end{bmatrix}$$
		the $3\times 3$ matrix representing the rotation about the second vector of an angle $\theta$. It can be observed in $N$
		that in $H_1(\SO(N);\Integral)$,
		\begin{equation}\label{piecewise_identity_c_N}
		\sum_{\textrm{rotating }\\i,j,k}
		\left([{\bold F}_{ijk}\xrightarrow{e_{ij}}(-{\bold F}_{jik})]+[(-{\bold F}_{jik})\xrightarrow{B(\theta_{jki}t)}{\bold F}_{jki}]\right)
		= c.
		\end{equation}
		One way to see this is by a continuation argument. The left-hand side is invariant if the vertices $(n_i,n_j,n_k)$ of $\Delta_{ijk}$
		vary continuously in $N$. Shrinking all the edge lengths towards zero, the sum of the external angles of $\Delta_{ijk}$ tends to $2\pi$.
		In the limit,
		the cycle of frames representing by the left-hand side is a full rotation of a frame, say ${\bold F}_{ijk}$, about its second vector,
		yielding the homology class $c\in H_1(\SO(N);\Integral)$.
		Therefore the identity (\ref{piecewise_identity_c_N}) holds in general.
		
		By applying the homomorphism ${\bold j}_*$, we have in $H_1(\SO(M);\Integral)$:
		\begin{equation}\label{piecewise_identity_c_M}
		\sum_{\textrm{rotating }\\i,j,k}
		\left([{\bold F}_{ijk}^M\xrightarrow{e_{ij}^M}(-{\bold F}_{jik}^M)]+[(-{\bold F}_{jik}^M)\xrightarrow{B(\theta_{jki}t)}{\bold F}_{jki}^M]\right)
		= c.
		\end{equation}
		
			
		Using the nearly totally geodesic annulus of Part (1),
		the left-hand side of the identity (\ref{piecewise_identity_c_M}) can be identified with
		the left-hand side of the identity (\ref{identity_c}).
		To visualize a homotopy between the loops of frames,
		assume for the moment that the annulus was genuinely totally geodesic.
		By foliating the annulus with smooth curves in its interior, the homotopy can be seen as
		the field of frames over the annulus, such that at any point,
		the first vector is tangent to the leaf, and the second vector
		is perpendicular to the annulus.
		The continuity near the cornered boundary can be ensured
		as the frames are parallel transported along the edges
		and rotated counterclockwise about the (instant) second vector
		of the external angle by the $B$-matrix at the corners. In our actual situation,
		the annulus is nearly totally geodesic
		so a homotopy should be given as an approximation to the above picture.
		With the picture in mind,
		it is possible to check the equality explicity by dividing the annulus
		into geodesic $2$-simplices, as we have done, and
		discretize the described homotopy.
		We omit the details since the procedure is straightforward.
		
		Therefore, we have verified the identity (\ref{identity_c}),
		and hence the claimed vanishing of $\sigma(\gamma_{ijk})$.		
	\end{proof}
	
	Corresponding to each $2$-simplex $\Delta_{ijk}$ spanned by a triple of vertices $(n_i,n_j,n_k)$
	of $N$, only for those with $i<j<k$, we take the corresponding oriented good curve $\gamma_{ijk}$ of $M$
	as guaranteed by Lemma \ref{noObstruction}.
	By Part (1) of Lemma \ref{noObstruction}, we may enrich $\gamma_{ijk}$ with a point $p\in\gamma_{ijk}$ and an auxiliary frame
	${\bold F}_p=(\vec{v}_p,\vec{n}_p)$.
	To be specific, we take $\vec{v}_p$ to be the direction vector of $\gamma_{ijk}$ at $p$,
	and require that $\vec{n}_p$ is $\frac{10\delta}L$--close to
	the parallel transport of the $\vec{n}$-vector of ${\bold F}_{ijk}^M$
	along a shortest possible path on	the asserted immersed nearly totally geodesic annulus.
	By Part (2) of Lemma \ref{noObstruction}, we invoke Theorem \ref{bounding} to construct an immersed
	oriented nearly totally geodesic $(R,\epsilon)$-panted subsurface bounded by $\gamma_{ijk}$,
	requiring the normal vector there to be $\frac{\delta}{L}$-close to $\vec{n}_p$.
	Moreover, the panted subsurface can be constructed to be fat
	in the sense that every properly embedded essential arc of the panted subsurface
	has length at least $R$ under the immersion into $M$,
	(see \cite[Section 3.1]{Sun1} and \cite[Lemma 4.2]{Liu}).
	Glue up the constructed panted subsurface with the annulus of Lemma \ref{noObstruction} (1)
	along $\gamma_{ijk}$. We denote the resulting subsurface as
		$$S_{ijk}^M\looparrowright M.$$
	This is our claimed counterpart in $M$ of the $2$-simplex $\Delta_{ijk}$ of $N$.
	
	We attach the abstract surfaces $S_{ijk}^M$ to $Z^1$ by the obvious attaching map of
	$\partial S_{ijk}^M=e_{ij}^M\cup e_{jk}^M \cup e_{ki}^M$.
	The resulting abstract $2$-complex is our claimed $Z$, coming with a naturally induced immersion:
		$$Z\looparrowright M.$$
		
	\subsubsection{Quasiconvexity and convex core}
	We verify that the immersion $Z\looparrowright M$ induces an embedding
	$\pi_1(Z)\to\pi_1(M)$,
	and describe the topology of the cover of $M$ corresponding to
	(any conjugate of) $\pi_1(Z)$.
	For the latter, we introduce a compact $3$-manifold $\mathcal{Z}$
	as a topological model, which can be constructed as follows:
	
	Note that the geodesic triangulation of $N$ induces a topological handle-decomposition
	of $N$ of which the simplices correspond naturally to the handles.
	Take the union $\mathcal{N}^{(1)}$ of the $0$-handles
	and the $1$-handles, (which is just a regular neighborhood of the $1$-skeleton of $N$).
	The $2$-handles $\Delta_{ijk}\times[0,1]$ are attached to $\partial\mathcal{N}^{(1)}$
	along disjoint annuli $\partial\Delta_{ijk}\times[0,1]$.
	We copy $\mathcal{N}^{(1)}$ to be $\mathcal{Z}^1$ and attach to $\partial\mathcal{Z}^1$
	analogue $2$-handles $S_{ijk}^M\times[0,1]$ along $\partial S_{ijk}^M\times[0,1]$ accordingly.
	Since there is a natural combinatorial identification of $\partial S_{ijk}^M$ with $\partial\Delta_{ijk}$
	by our construction of $S_{ijk}^M$, the attaching map can be defined via the identification.
	The result is an abstract compact $3$-manifold $\mathcal{Z}$ with a spine homeomorphic to
	the $2$-complex $Z$.
	
	Recall that for a Kleinian group $\Gamma$ acting on $\mathbb{H}^3$, the convex hull of $\Gamma$
	is the convex hull of the limit set in $\mathbb{H}^3$, and the convex core of $\mathbb{H}^3/\Gamma$
	is the quotient of the convex hull by the action of $\Gamma$.
	
	\begin{lemma}\label{injectivity}
		The immersion $Z\looparrowright M$ induces an embedding	$\pi_1(Z)\to\pi_1(M)$
		whose image is quasiconvex.
		Moreover, the convex core of
		the cover of $M$ corresponding to $\pi_1(Z)$
		is homeomorphic to the compact $3$-manifold $\mathcal{Z}$.
	\end{lemma}
	
	\begin{proof}
		The proof is almost completely the same as the proof of \cite[Theorem 3.8]{Sun2}.
		To be more specific, our assumptions in the construction of $S_{ijk}^M$ about the normal vectors and
		the absence of short essential arcs ensure that $Z\looparrowright M$ is $\pi_1$--injective
		and quasiconvex. The cover of $M$ corresponding to $\pi_1(Z)$ is isometric to
		the open hyperbolic $3$-manifold $\mathbb{H}^3/\pi_1(Z)$,
		via the induced holonomy representation.
		For sufficiently small $\epsilon$ and sufficiently large $R$ respecting $\epsilon$,
		the representation can be deformed from a model one
		of which the surface pieces $S_{ijk}^M$ are genuinely totally geodesic and the convex core
		is homeomorphic to $\mathcal{Z}$.
		To adapt the proof of \cite[Theorem 3.8]{Sun2}, we only need to be careful about
		a few quantitative estimates related to dihedral angles there.
		The validity of those estimates essentially relied on
		a uniform bound of the dihedral angles away from $0$,
		which was definite and explicitly provided.
		In our case, the uniform bound depends on our initial choice of the geodesic triangulation of $N$,
		and is in fact the constant $\phi_0\in(0,\pi)$.
		After modifying the estimates accordingly,
		we can repeat the proof of \cite[Theorem 3.8]{Sun2} (see \cite[Section 4]{Sun2} for the argument and discussions)
		for our case with only notational changes.
	\end{proof}
	
	We close the subsection by pointing out the range of the parameters $(\delta,L)$ and $\epsilon$ that we need
	for the above construction.
	After fixing a choice for the set of initial data,
	we require that a constant $0<\epsilon<1$ to be chosen
	less than $\mathrm{inj}(M)$ and $10^{-1}\phi_0$, and take $\delta$ to be $10^{-1}\epsilon$.
	The sufficiently large constant $L>1$ can be chosen,
	according to $\delta$ and the geometry of $M$ and our initial data,
	so that the application of the good pants constructions
	(Theorems \ref{bounding} and \ref{connectionprinciple}) are valid.
	Note that, overall, we have only invoked
	Theorems \ref{bounding} and \ref{connectionprinciple} for a definite finite number of times,
	which is determined by the triangulation of $N$.

\subsection{Virtual embedding and 1-domination}
	Provided with the immersed quasiconvex $2$-complex $Z\looparrowright M$, we denote by
		$$\tilde{M}_Z\to M$$
	the cover of $M$ corresponding to (a conjugate of) the subgroup $\pi_1(Z)$ of $\pi_1(M)$,
	as guaranteed by Lemma \ref{injectivity}.
	The convex core of $\tilde{M}_Z$ is an embedded convex $3$-submanifold
	which is homeomorphic to $\mathcal{Z}$ as described by Lemma \ref{injectivity}.
	Because of the quasiconvexity of $\pi_1(Z)$ and convexity of the hyperbolic metric,
	the $r$-neighborhood of the convex core of $\tilde{M}_Z$ for any radius $r$,
	denoted by
		$$K_r\subset\tilde{M}_Z,$$
	is again convex and homeomorphic to $\mathcal{Z}$.
		
	By the LERF property of finite-volome hyperbolic $3$-manifold groups,
	proved by I.~Agol \cite{Ag} based on the work of D.~Wise \cite{Wi},
	$\pi_1(Z)$ is a separable subgroup of $\pi_1(M)$.
	Using Scott's topological characterization of the LERF property \cite{Sc},
	for any given $r$,
	the infinite cover $\tilde{M}_Z$ factors through a finite cover $M'$ of $M$, namely,
		$$\tilde{M}_Z\to M'\to M,$$
	such that the projection induces an embedding of $K_{r+2}$ into $M'$.
	For our application, when a positive constant $\kappa$ is given,
	it is safe enough to produce the finite cover $M'$
	taking $r$ to be $\kappa^{-1}(1+\mathrm{Diam}(N))$,
	or simply $(1+\mathrm{Diam}(N))$ if $\kappa$ exceeds $1$.
		
	\begin{lemma}\label{one_domination}
		There exists a degree-$1$	map $f\colon M'\rightarrow N$.
		When a positive constant $\kappa$ is given and the initial data is prepared accordingly,
		$f$ can be required to be $\kappa$-Lipschitz.
	\end{lemma}
	
	\begin{proof}
		Topologically, $f$ can be constructed,
		with respect to any fixed constant $r>0$, as follows.
		With the notations in the definition of $\mathcal{Z}$,
		observe the following proper degree-$1$ map
		$\mathcal{Z}\to \mathcal{N}^{(2)}$:
		The restricted map $\mathcal{Z}^{(1)}\to\mathcal{N}^{(1)}$
		is the identifying homeomorphism, and the restricted maps
		$S_{ijk}^M\times[0,1]\to \Delta_{ijk}\times[0,1]$
		are induced by pinching the first factor.
		Identifying $K_1$ with $\mathcal{Z}$ via a homeomorphism,
		we obtain a proper degree-$1$ map
			$$f_1\colon K_1\to \mathcal{N}^{(2)},$$
        Since the complementary components of $\mathcal{N}^{(2)}$ are the $3$-handles of $N$,
		the defined map $f_1$ can be extended over every complementary component of $K_1$ in $M'$.
		More precisely, we first require every component of
		$K_{2}\setminus K_1\cong \partial K_1\times [1,2]$
		to be mapped onto a $3$-handle $D_i\cong S^2\times[0,1]\,/\,S^2\times\{1\}$ of $N$,
		as induced by $f_1|_{\partial K_1}$. So we get a map $$f_2\colon K_2\to N,$$ which maps
        $\partial K_{2}$ to centers of $3$-handles of $N$.
		Then we require every component of $M'\setminus K_{r+2}$
		to be collapsed to an arbitrary point in $N$, and every component $\Sigma\times [2,r+2]$ of
        $K_{r+2}\setminus K_2\cong \partial K_2\times [2,r+2]$ to be mapped to a path in $N$ from the point
        corresponding to the image of $\Sigma\times \{2\}$ to the point corresponding to the image of $\Sigma\times \{r+2\}$,
        via a map collapsing each $\Sigma\times\{*\}$ to a point.
		The resulting map
			$$f\colon M'\to N$$
		is hence a degree-$1$ map as claimed.
		
		When a positive constant $\kappa$ is given and the initial data is prepared accordingly,
		the above map $f$ can be metricized to be a $\kappa$-Lipschitz one.
		In fact, we may assume without loss of generality that $\kappa$ is at most $1$,
		otherwise replace it with $1$.
		Fix a constant $r>\kappa^{-1}(1+\mathrm{Diam}(N))$, which is at least $1$.	
		The geodesic triangulation of $N$ has all the edges $(10^{-2}\kappa)$--short, whereas
		$K_1$ has a nearly totally geodesic spine, homeomorphic to our $2$-complex $Z$.
		The segments of the spine are lifts $e_{ij}^M$, so they all have length at least $1$;
		the $2$-pieces are lifts of $S_{ijk}^M$, so they have no essential arcs shorter than $1$.
		Moreover, we can ensure that $K_1$ contains the $1$-neighborhood of the spine.
		With the picture in mind, the map $f_1$ can be constructed to be $\kappa$-Lipschitz,
		mapping $K_1$ properly onto the $(10^{-3}\kappa)$--neighborhood of the $2$-skeleton of $N$.
        Since the $3$-handles in $N$ have diameter at most $10^{-2}\kappa$, the map $f_2\colon K_2\to N$
        is $\kappa$-Lipschitz on $K_2\setminus K_1$. Since $r=\kappa^{-1}(1+\mathrm{Diam}(N))$,
        and we can assume the paths of the image of $\partial K_2\times [2,r+2]$ have length at most $\mathrm{Diam}(N)$,
		the described extension of $f_2$ can also be realized to satisfy the $\kappa$-Lipschitz condition
		in every complementary component of $K_2$ in $M'$.
		By the construction of $M'$ (or the choice of $r$),
		it follows easily that the map $f$ constructed above is globally $\kappa$--Lipschitz.
		We have omitted some estimation details in the course,
		but those should be conceptually clear and straightforward to check.
	\end{proof}
	
	This completes the proof of Theorem \ref{virtual_one_domination_hyperbolic}.

	\subsection{Further remark}
	A closer examination of the proof of Theorem \ref{virtual_one_domination_hyperbolic}
	reveals that the essential information provided by the geometric triangulation
	is a compatible assignment of dihedral angles to the tetrahedra.
	To be precise,
	for any closed orientable $3$-manifold $N$ with a simplicial triangulation,
	the same construction works
	if there exists an assignment of dihedral angles to every tetrahedron of $N$
	subject to the following requirement:
	For each link of a vertex, the induced angle assignment to the $2$-simplices can be realized
	by a spherical geometric structure of that vertex link.

\section{Virtual 1-domination onto arbitrary 3-manifolds}\label{Sec-generalCase}

	In this section, we derive the main results of the paper, namely,
	Theorems \ref{main} and \ref{contraction}, from Theorem \ref{virtual_one_domination_hyperbolic}.
	
	To derive the topological virtual $1$-domination, Theorem \ref{main},
	we invoke the fact that every closed oriented $3$-manifold
	is $1$-dominated by a hyperbolic closed oriented $3$-manifold, see Boileau--Wang \cite[Proposition 3.3]{BW}.
	In particular, given any $M$ hyperbolic and $N$ arbitrary as assumed by Theorem \ref{main},
	we may take a hyperbolic $N'$ which $1$-dominates $N$ via a map $N'\to N$,
	and construct by Theorem \ref{virtual_one_domination_hyperbolic}
	a virtual $1$-domination $M'\to N'$ using some finite cover of $M'$.
	The composed map $M'\to N'\to N$ has degree $1$, so $M$ virtually $1$-dominates $N$.
	This completes the proof of Theorem \ref{main}.
	
	To derive the metric-contracting virtual $1$-domination
	suppose that $M$ is hyperbolic and
	$N$ is Riemannian as assumed by Theorem \ref{contraction}.
	By invoking \cite[Proposition 3.3]{BW}, again we find a hyperbolic $N'$ which $1$-dominates $N$.
	We take a Lipschitz degree-$1$ map $N'\to N$, whose Lipschitz constant is denoted by $\lambda>0$.
	When any constant $\kappa>0$ is given, we apply Theorem \ref{virtual_one_domination_hyperbolic}
	with respect to the constant $\lambda^{-1}\kappa$, obtaining a $(\lambda^{-1}\kappa)$--Lipschitz
	virtual $1$-domination $M'\to N'$.
	The composed map $M'\to N'\to N$ is hence $\kappa$-Lipschitz and degree-$1$ as claimed.
	This completes the proof of Theorem \ref{contraction}.

\end{document}